\documentclass{amsart}

\usepackage{amsmath, amsthm, amssymb}
\usepackage[colorlinks=true,linkcolor=blue,urlcolor=blue,unicode=true]{hyperref}
\usepackage[capitalize,noabbrev]{cleveref}
\usepackage{tikz-cd}

\numberwithin{equation}{section}

\crefname{prop}{Proposition}{Propositions}
\crefname{equation}{Diagram}{Diagrams}

\theoremstyle{plain}
\newtheorem{thm}{Theorem}[section]
\Crefname{thm}{Theorem}{Theorems}
\newtheorem{cor}[thm]{Corollary}
\newtheorem{lem}[thm]{Lemma}
\newtheorem{prop}[thm]{Proposition}
\theoremstyle{definition}
\newtheorem{defi}[thm]{Definition}
\newtheorem{rem}[thm]{Remark}

\newcommand{\R}{\mathbb{R}}

\newcommand{\HLC}{\mathrm{HLC}}

\newcommand{\indexSet}{\mathcal{I}}
\renewcommand{\H}{H}

\DeclareMathOperator{\im}{im}

\DeclareMathOperator{\coker}{coker}
\DeclareMathOperator{\St}{St}
\DeclareMathOperator{\lcs}{lcs}
\newcommand{\CH}{\check{H}}
\newcommand{\PFD}{\mathrm{PFD}}

\DeclareMathOperator{\Cov}{Cov}

\DeclareMathOperator{\Vietoris}{Vtr}

\newcommand{\Vect}{\mathsf{Vect}}

\newcommand{\sing}{S}
\newcommand{\viet}{V}

\title[Comparison of persistent singular and \v{C}ech homology]{Comparison of persistent singular and \v{C}ech homology for locally connected filtrations}
\author{Maximilian Schmahl}
\address{Maximilian Schmahl, Mathematisches Institut, Universit\"at Heidelberg}
\email{\href{mailto:mschmahl@mathi.uni-heidelberg.de}{mschmahl@mathi.uni-heidelberg.de}}
\date{\today}

\subjclass[2020]{55N05, 55N10, 55N31}

\AtBeginDocument{%
   \def\MR#1{}
}
\begin{document}
\begin{abstract}
    We show that the interleaving distance between the persistent singular homology and the persistent \v{C}ech homology of a homologically locally connected filtration consisting of paracompact Hausdorff spaces is 0.
\end{abstract}

\maketitle
\section{Introduction}
Originating from topological data analysis, there has recently been considerable interest in the development of \emph{persistence theory} and in particular \emph{persistent homology}, which is the study of the homological evolution of a \emph{filtration} of spaces. 
Specifically, given a space $X$ and a function $f \colon X \to \R$, which need not be continuous, we obtain a filtration of $X$ by the sublevel sets $f_{\leq t} = f^{-1}(-\infty, t]$ for $t \in \R$. 
Applying a suitable homology theory $\H$ to each sublevel set and the inclusions between them, we obtain a \emph{persistence module} $\H(f_{\leq \bullet})$, i.e., a functor $\mathbb{R} \to \Vect$, where we consider the real numbers as a category via their poset structure.
In favorable cases, persistence modules admit complete invariants called \emph{barcodes} or \emph{persistence diagrams}.
These invariants may then be used to study the original function, aided by the fact that passing from functions to persistence diagrams is a Lipschitz map with respect to suitable metrics.

Given that persistent homology can be used for analyzing sublevel sets of functions, it might not come as a surprise that it can be intimately linked to Morse theory. 
In fact, many of the modern concepts in persistence theory have early precursors in Morse's framework of \emph{functional topology} \cite{Morse.1937, Morse.1938, Morse.1939, Morse.1940}, which he used to study minimal surfaces. 
The connection between persistence theory and functional topology has recently been elaborated on in \cite{Bauer.2022}, with a central theme being that interesting finiteness properties of persistent homology may be ensured by homological local connectedness conditions on the given filtration.

Classically, a single space $X$ is called \emph{(co)homologically locally connected} (or $\HLC$) if for any $x \in X$ and any neighborhood $V$ of $x$ there is a neighborhood $U$ of $x$ with $U \subseteq V$ such that the inclusion $U \to V$ induces the trivial map in (co)homology.
There are several results stating that the homology of a compact Hausdorff space has finite rank if it is locally connected in some sense as above, see for example \cite[Corollary II.17.7]{Bredon.1997}.
This immediately implies that the persistent homology of the sublevel set filtration of a function $f$ with locally connected compact Hausdorff sublevel sets is \emph{pointwise finite dimensional} (or $\PFD$), meaning that $\H(f_{\leq t})$ is finite dimensional for all $t$, which implies that it admits a barcode by Crawley-Boevey's Theorem \cite{Crawley-Boevey.2015, MR4143378}.
However, the assumption that every sublevel set is $\HLC$ is for example not satisfied in Morse's minimal surface setting, while the following weaker one is \cite{Bauer.2022}.

\begin{defi}\label{defi:hlc}
    Let $f \colon X \to \R$ be a function on a topological space. We say that the sublevel set filtration $f_{\leq \bullet}$ induced by $f$ is \emph{homologically locally connected ($\HLC$)} with respect to the homology theory $\H$ if for any $x \in X$, any neighborhood $V$ of $x$, and any pair of indices $s,t$ with $f(x) < s < t$ there is a neighborhood $U$ of $x$ with $U \subseteq V$ such that the inclusion $f_{\leq s} \cap U \to f_{\leq t} \cap V$ is taken to the trivial map by $\H$.  
\end{defi}

By \cite{Bauer.2022}, we know that if the sublevel set filtration of $f$ is $\HLC$ (a slightly weaker condition actually suffices) and all sublevel sets are compact Hausdorff, then the induced map $\H(f_{\leq s}) \to \H(f_{\leq t})$ has finite rank whenever $s < t$. 
Persistence modules with this finiteness property are called \emph{q-tame}, and they are of great interest since they constitute one of the largest known classes of persistence modules that do admit persistence diagrams \cite{Chazal.2016a,Chazal.2016b}. 

Apart from finiteness results as the above ones, $\HLC$ conditions play an important role in comparison results for different homology theories, see \cite{Skljarenko.1980} for a plethora of results in this direction.
As a specific instance, Marde\v{s}i\'{c} \cite{Mardesic.1958} proved that paracompact Hausdorff spaces that are locally connected with respect to singular homology have naturally isomorphic singular and \v{C}ech homology groups.
Similarly to before, applying this to each sublevel set individually implies that the persistent singular and \v{C}ech homologies of a filtration are naturally isomorphic if the constituent sets are all $\HLC$ paracompact Hausdorff spaces.
The goal of this paper is now to provide a generalization of Mardešić’s Theorem in the same vein as the main theorem from \cite{Bauer.2022} for filtrations that only satisfy the $\HLC$ condition as in \cref{defi:hlc}. 

To state the result, we first recall that the \emph{category of persistence modules} is simply the category of functors $\R \to \Vect$, so that a \emph{morphism of persistence modules} $\phi \colon M \to N$ is a natural transformation, i.e., a collection of maps $\phi_t \colon M_t \to N_t$ for $t \in \R$ such that $\phi_t \circ M_{s,t} = N_{s,t} \circ \phi_s$ whenever $s \leq t$.
This category is abelian with kernels, cokernels, etc.\@ given by their pointwise counterparts, i.e., we have $(\ker\phi)_t = \ker(\phi_t)$ and $(\coker\phi)_{t} = \coker(\phi_t)$ for all $t \in \R$.
A morphism of persistence modules is called a \emph{weak isomorphism} if its kernel and cokernel are \emph{ephemeral}, where a persistence module $M$ is called ephemeral if the structure map $M_{s,t} \colon M_s \to M_t$ is trivial whenever $s < t$.
A weak isomorphism may be thought of as weakening the usual notion of isomorphism of persistence modules in the same way as the q-tameness condition weakens the $\PFD$ condition.

\begin{thm}\label{thm:main}
    If $f \colon X \to \R$ induces a filtration of paracompact Hausdorff spaces that is locally connected with respect to singular homology with coefficients in an abelian group $G$, then the natural map $\varphi \colon H_*(f_{\leq \bullet};G) \to \CH_*(f_{\leq \bullet};G)$ from its persistent singular to its persistent \v{C}ech homology is a weak isomorphism.
\end{thm}

The construction of the natural map $\varphi$ from singular to \v{C}ech homology will be reviewed later on. 
Note that the persistent homologies appearing in the theorem need not be persistence modules in the sense that we have talked about before, because for general abelian coefficient groups $G$ they will be functors from $\R$ to the category of $G$-modules and not necessarily to a category of vector spaces.
However, the notion of weak isomorphism still makes sense in this more general setting and the result holds as stated.

We have mentioned previously that passing from filtrations to persistence diagrams is Lipschitz with respect to appropriate distances, namely the supremum norm for functions and the bottleneck distance for barcodes. 
The corresponding distance on the intermediate level of persistence modules is the \emph{interleaving distance}.
Given two diagrams $M,N \colon \R \to \mathsf C$, where $\mathsf{C}$ is any category, a $\delta$-interleaving between $M$ and $N$ consists of a pair of natural transformations $(M_t \to N_{t+\delta})_t$ and $(N_t \to M_{t+\delta})_t$ from one diagram to a shifted version of the other and vice versa, which compose to the structure maps $(M_t \to M_{t+2\delta})_t$ and $(N_t \to N_{t+2\delta})_t$.
Clearly, the case $\delta=0$ describes an isomorphism, and the infimum $\delta \geq 0$ for which $M$ and $N$ admit a $\delta$-interleaving is defined as the interleaving distance $d_I(M,N)$ between them.
The interleaving distance between two persistence modules is 0 if there is a weak isomorphism from one to the other, so we obtain the following consequence of \cref{thm:main}.

\begin{cor}\label{cor:main}
     If $f \colon X \to \R$ induces a filtration of paracompact Hausdorff spaces that is locally connected with respect to singular homology with coefficients in an abelian group $G$, then
     \[
        d_I \big( H_*(f_{\leq \bullet};G), \CH_*(f_{\leq \bullet};G) \big) = 0.
     \]
\end{cor}

A persistence module that has 0 interleaving distance to a q-tame persistence module is again q-tame, so combining our results with \cite{Bauer.2022}, we obtain that a compact filtration that is $\HLC$ with respect to singular homology is not only q-tame with respect to singular homology, but also with respect to \v{C}ech homology. 
Being q-tame, both persistence modules then admit a persistence diagram, and as a consequence of \cref{cor:main} the two persistence diagrams agree.
%This ties into Morse's insistence on using \v{C}ech homology instead of the better-known singular theory throughout his work on functional topology, which he justifies with the fact that \v{C}ech homology has certain continuity properties that singular homology does not have.
%While these continuity properties are also of interest from the point of view of persistence theory \cite{Schmahl.2021}, our main result implies that, when employing the persistence-based approach to Morse inequalities from \cite{Bauer.2022}, the use of \v{C}ech homology might not be necessary.

Our techniques will allow us to prove a stronger, quantitative version of \cref{cor:main}. It bounds the interleaving distance between singular and \v{C}ech homology in terms of a measurement for how locally connected a filtration is.

\begin{defi}
    Let $f \colon X \to \R$ be a function on a topological space. For $\delta \geq 0$, we say that the sublevel set filtration $f_{\leq \bullet}$ induced by $f$ is \emph{$\delta$ homologically locally connected ($\delta-\HLC$)} with respect to the homology theory $\H$ if for any $x \in X$, any neighborhood $V$ of $x$, and any pair of indices $s,t$ with $f(x) < s \leq s + \delta < t$ there is a neighborhood $U$ of $x$ with $U \subseteq V$ such that the inclusion $f_{\leq s} \cap U \to f_{\leq t} \cap V$ is taken to the trivial map by $\H$.
    We define the \emph{local connectedness shift of $f$ with respect to $\H$} as
    \[
    \lcs_{\H}(f) = \inf \{ \delta > 0 \mid f \text{ is }\delta-\HLC \}.
    \]
\end{defi}

If the filtration induced by a function $f$ is $\HLC$, then its local connectedness shift is $0$. Thus, the following result presents another generalization of \cref{cor:main}.

\begin{thm}\label{thm:main_strong}
    If $f \colon X \to \R$ induces a filtration of paracompact Hausdorff spaces and $G$ is an abelian group, then
     \[
        d_I \big( H_{\leq d-1}(f_{\leq \bullet};G), \CH_{\leq d-1}(f_{\leq \bullet};G) \big) \leq d \cdot \lcs (f) 
     \]
     for all $d$, where $\lcs (f)$ is the local connectedness shift of $f$ with respect to singular homology with coefficients in $G$.
\end{thm}

As mentioned before, for a fixed homology theory, the well-known \emph{stability theorem} \cite{Cohen-Steiner.2007,Chazal.2009,Bauer.2015} asserts that the passage from functions to barcodes of their persistent sublevel set homology is 1-Lipschitz with respect to the supremum norm and the \emph{bottleneck distance}.
For $\PFD$ persistence modules, the passage to barcodes is an isometry with respect to the interleaving and bottleneck distances, respectively.
Hence, \cref{thm:main_strong} can be interpreted as providing a kind of stability result for when the function is fixed but the homology theory varies from singular to \v{C}ech.
Concerning stability with respect to variations of the function, we will also prove that the local connectedness shift $\lcs_{H}$ is 2-Lipschitz for any homology theory $H$ (\cref{prop:lcs_lipschitz}).

Both \cref{thm:main,thm:main_strong} will be proven by constructing approximate inverses for the natural map from singular to \v{C}ech homology.
More specifically, we will show that if $f$ induces a $\delta-\HLC$ filtration then there are natural maps $\psi_{s,t}$ for indices $s$ and $t$ with $t - s > d \cdot \delta$ such that the diagram
\begin{equation}\label{diagram:main}
    \begin{tikzcd}
        H_{\leq d-1}(f_{\leq s};G) \arrow[r] \arrow[d, "\varphi_s"] & H_{\leq d-1}(f_{\leq t};G)\arrow[d,"\varphi_t"] \\
       \CH_{\leq d-1}(f_{\leq s};G) \arrow[r] \arrow[ru,"\psi_{s,t}"] & \CH_{\leq d-1}(f_{\leq t};G)
    \end{tikzcd}
\end{equation}
commutes.
The constructions of the maps $\varphi$ and $\psi$ will happen on the chain level and follow the arguments presented in \cite{Mardesic.1958}.

Other approaches to proving our results could be to use an approximate nerve theorem \cite{Govc.2018} and analyze how the interleaving distance behaves with respect to inverse limits in the category of persistence modules, or to use an approach via cosheaves as in \cite{Bredon.1997}, akin to the usual approach to comparison results between singular and \v{C}ech \emph{co}homology via sheaf cohomology.
In the present work, we do not investigate these possibilities further.

\subsection*{Outline}
We will recall the construction of the natural map $\varphi$ from singular to \v{C}ech homology in \cref{section:singular_to_cech}.
In \cref{section:cech_to_singular}, we will construct the maps $\psi_{s,t}$ from \v{C}ech to singular homology for suitable indices $s < t$.
Next, we prove commutativity of \cref{diagram:main} by constructing two chain homotopies in \cref{section:homotopy_singular,section:homotopy_cech}.
We use these results to prove \cref{thm:main,thm:main_strong} in \cref{section:final_proof}.
Finally, we show that the local connectedness shift is a 2-Lipschitz map in \cref{sec:lcs_lipschitz}.

\subsection*{Notation and Conventions}
Throughout, we will fix an abelian coefficient group $G$ for all chain complexes and homology groups, which will not be referenced in the notation anymore.

\section{The map from singular to \v{C}ech homology}
\label{section:singular_to_cech}
We recall the definition of \v{C}ech homology in terms of Vietoris complexes of covers, how singular homology may also be defined depending on a cover, and how this leads to a natural map from singular to \v{C}ech homology.

A \emph{cover} of a topological space $X$ is a set $\alpha$ consisting of open subsets of $X$ such that $X = \bigcup_{U \in \alpha} U$.
The set of all covers will be denoted by $\Cov(X)$. 
It is a directed set with respect to the \emph{refinement} relation, where a cover $\alpha$ of a space $A \subseteq X$ is said to \emph{refine} a cover $\beta$ of $X$ if for all $U \in \alpha$ there exists $V \in \beta$ with $U \subseteq V$.
For a cover $\alpha$, we define its \emph{Vietoris complex} to be the simplicial complex given by
\[
    \Vietoris (\alpha) = \left\{ \rho \subseteq X \mid \rho \text{ is finite and } \rho \subseteq U \text{ for some } U \in \alpha \right\}.
\]
We write $\Vietoris^d(\alpha)$ for the subcomplex of simplices with dimension at most $d$.
We will write $\viet_{*}(\alpha)$ for the simplicial chain complex of $\Vietoris(\alpha)$ and $\viet^{d}_{*}(\alpha)$ for the simplicial chain complex of $\Vietoris^{d}(\alpha)$.
Simplices in $\Vietoris(\alpha)$ given by points $v_{0},\dots,v_{n} \in U \in \alpha$ will be denoted by $\{v_{0},\dots,v_{n}\}$, and \emph{oriented} simplices in $\viet_{*}(\alpha)$ by $[v_{0},\dots,v_{n}]$.
If $\alpha$ refines $\beta$, there is an obvious simplicial map $\pi \colon \Vietoris(\alpha) \to \Vietoris(\beta)$, and after applying simplicial homology we obtain an inverse system of abelian groups.
The \emph{\v{C}ech homology} of $X$ is then defined as the limit of this diagram, i.e.,
\[
    \CH_*(X) \ = \lim_{\alpha \in \Cov(X)} H(\viet_{*}(\alpha)).
\]
Note that this is not the classical definition of \v{C}ech homology as presented for example in \cite{Eilenberg.1952}, but the definition above agrees with the classical one by Dowker's Theorem \cite{Dowker.1952}.

\begin{rem}
    The definition of \v{C}ech homology in terms of Vietoris complexes that we use works particularly well in conjunction with the $\HLC$ condition in terms of refining neighborhoods as presented in the introduction.
    There is an equivalent way of stating the $\HLC$ condition in terms of intersections of neighborhoods \cite{DePauw.2007} akin to the notion of good cover that is used for Nerve Theorems \cite{Govc.2018,Bauer.2022a}. 
    Morally, this alternative characterization of local connectedness works better in conjunction with the classical definition of \v{C}ech homology in terms of nerves of covers.
\end{rem}

Denoting by $\sing_*(X)$ the singular chain complex of $X$, we define $\sing_{*}(\alpha)$ for some cover $\alpha$ of $X$ to be the subcomplex of $\sing_*(X)$ generated by those singular simplices whose image is contained in some set $U \in \alpha$. 
The union of all images of singular simplices making up a singular chain will be called the \emph{support} of the chain.
If $\alpha$ refines $\beta$ there is an inclusion of chain complexes $\eta \colon \sing_*(\alpha) \to \sing_{*}(\beta)$.
The inclusion $\sing_*(\alpha) \to \sing_*(\{X\}) = \sing_*(X)$ can be shown to induce an isomorphism in homology for all $\alpha \in \Cov(X)$ \cite{Eilenberg.1952}, with inverse given by a subdivision construction. In particular, we get an isomorphism
\[
    H_*(X) \ \cong \lim_{\alpha \in \Cov(X)} H(\sing_{*}(\alpha)),
\]
where $H_*(X)$ denotes the singular homology of $X$.

Let $e_0, \dots, e_n$ be the vertices of the standard $n$-simplex. 
For any $\alpha \in \Cov(X)$, we obtain a chain map $\mu \colon \sing_*(\alpha) \to \viet_{*}(\alpha)$ by sending a singular $n$-simplex $\sigma$ to the oriented Vietoris simplex $[\sigma(e_0), \dots, \sigma(e_n)]$ if the $\sigma(e_i)$ are all distinct and to 0 otherwise.
These chain maps are natural with respect to refinement, so they give rise to a map
\[
    H_*(X) \cong \lim_{\alpha \in \Cov(X)} H(\sing_{*}(\alpha)) \to \lim_{\alpha \in \Cov(X)} H(\viet_{*}(\alpha)) \cong \CH_*(X),
\]
which can be shown to also be natural in the argument $X$. 
We obtain a map 
$
    \varphi \colon H_*(f_{\leq \bullet}) \to \CH_*(f_{\leq \bullet})
$
from the persistent singular to the persistent \v{C}ech homology of the sublevel set filtration of a real-valued function $f$.

\section{A map from \v{C}ech to singular homology}
\label{section:cech_to_singular}
We will now show how covers satisfying certain local triviality assumptions can be used to construct maps from \v{C}ech homology to singular homology in filtered spaces.

If $\alpha$ is a cover of a topological space $X$ and $U \in \alpha$, we define the \emph{star} of $U$ with respect to $\alpha$ as 
\[
\St_{\alpha} U = \bigcup_{\substack{U' \in \alpha \\ U \cap U' \neq \emptyset}} U' \subseteq X.
\]
A cover $\alpha$ is called a \emph{star refinement} of another cover $\beta$ if for all $U \in \alpha$ there exists $V \in \beta$ with $\St_{\alpha} U \subseteq V$.
We say that the star refinement is $\HLC$ with respect to $\tilde{H}$ if $V$ can be chosen such that the inclusion $\St_{\alpha} U \to V$ becomes trivial after applying $\tilde{H}$.

%Now, let $A \subseteq X$ be a subspace. 
%A set $\alpha$ consisting of open sets in $X$ is called a \emph{subspace cover} of $A$ in $X$ if $A \subseteq \bigcup_{U \in \alpha} U$.
%If $\alpha$ is a subspace cover of $A$ in $X$, we obtain a cover $\alpha \cap A \in \Cov(A)$ of $A$ defined as $\alpha \cap A = \{U \cap A \mid U \in \alpha \}$.
%If $\alpha$ is a cover of $A$ and $\xi$ is a cover of $X$, we say that $\alpha$ is a \emph{subspace star refinement} if for all $U \in \alpha$ there exists $V \in \xi$ with $\St_{\alpha} U \subseteq V$.

\begin{defi}\label{defi:admissibility}
    Let $A_0 \subseteq A_1 \subseteq \dots \subseteq A_d \subseteq X$ be topological spaces. 
    A sequence of covers $\alpha_i$ of $A_i$ is called \emph{admissible} if $\alpha_i$ is an $\HLC$ star refinement of $\alpha_{i+1}$ with respect to reduced singular homology for all $i$.
    Given an admissible sequence of subspace covers, a chain map $\lambda \colon \viet^d_*(\alpha_0) \to \sing_*(\alpha_d)$ is called \emph{admissible} if for any $0$-Vietoris simplex $\{v\}$, $\lambda(\{v\})$ is a singular $0$-simplex taking the value $v$ and if for any oriented $n$-simplex $\rho$ in $\viet_{*}^{d}(\alpha_{0})$ there exists $V \in \alpha_{n}$ such that the support of the singular $n$-chain $\lambda(\rho)$ and all vertices of $\rho$ are contained in $V$.
\end{defi}

For our later application to the sublevel set filtration of a function $f \colon X \to \R$, one may think of the sets $A_i$ in the definition above as sublevel sets of $f$. 

\begin{lem}\label{lem:cech_to_singular}
    Let $A_0 \subseteq A_1 \subseteq \dots \subseteq A_d \subseteq X$ be topological spaces and let $(\alpha_i)_i$ be an admissible sequence of covers. 
    Then there exists an associated admissible chain map $\lambda \colon \viet^d_*(\alpha_0) \to \sing_*(\alpha_d)$.
\end{lem}
\begin{proof}
    We define the desired map $\lambda$ on oriented Vietoris simplices and extend linearly. 
    We also proceed inductively in the dimension of the simplices.
    
    For a $0$-simplex $\{v\}$, we set $\lambda(\{v\})$ to be the singular simplex with value $v$.
    For an oriented $1$-simplex $[v_0,v_1] \in \viet^d_*(\alpha_0)$, we choose $U \in \alpha_0$ with $v_0, v_1 \in U$. 
    The cover $\alpha_0$ is an $\HLC$ star refinement of $\alpha_1$, so we may choose $V \in \alpha_1$ such that the inclusion $U \subseteq V$ is trivial in singular homology. 
    This means that the augmented singular $0$-cycle $c = \lambda(\partial([v_0,v_1])) =  \lambda([v_0]) - \lambda([v_1]) \in \sing_{*} (U)$ becomes a boundary when considered as a chain in $V$, so there exists a 1-chain $c' \in \sing_*(V) \subseteq \sing_{*}(\alpha_{1})$ such that $\partial(c') = c$.
    We set $\lambda([v_0,v_1]) = c'$.
    
    Now, inductively, assume that $\lambda$ has been defined for oriented $(n-1)$-simplices and consider an oriented $n$-simplex $\rho = [v_0,\dots,v_n]$, where $2 \leq n \leq d$.
    Let $\rho_i = [v_0,\dots,\hat{v_i},\dots,v_n]$ be the $i$-th boundary component of $\rho$.
    By our inductive assumption, we have singular $(n-1)$-chains $\lambda(\rho_i) \in \sing_{*}(\alpha_{d})$, and using admissibility we may choose sets $U_i \in \alpha_{n-1}$ such that the support of $\lambda(\rho_i)$, as well as the points $v_{j}$ for $j \neq i$ are contained in $U_i$.
    Because we assume $n \geq 2$ we can always choose $j \notin \{0, i\}$ to obtain $v_j \in U_0 \cap U_i$, so we must have $U_0 \cap U_i \neq \emptyset$ for all $i$.
    It follows that $U_i \subseteq \St_{\alpha_{n-1}} U_0$ so that the support of $\lambda(\rho_i)$ is contained in $\St_{\alpha_{n-1}} U_0$ for all $i$.
    Since $\alpha_{n-1}$ is an $\HLC$ star refinement of $\alpha_{n}$ we may choose $V \in \alpha_{n}$ such that $\St_{\alpha_{n-1}} U_0 \subseteq V$ and the corresponding inclusion map is trivial in singular homology.
    Hence, the $(n-1)$-cycle $c = \lambda(\partial(\rho)) =  \sum_i (-1)^i \lambda(\rho_i) \in \sing_{*}(\St_{\alpha_{n-1}} U_0)$ becomes a boundary when considered as a chain in $V$, i.e., there exists a singular $n$-chain $c' \in \sing_*(V) \subseteq \sing_{*}(\alpha_{n})$ such that $\partial(c') = c$.
    Setting $\lambda(\rho) = c'$ yields a chain map as desired.
\end{proof}

Given a filtration of spaces as above, there might be many choices of admissible subspace covers, and for each such choice there might again be many choices of admissible chain maps. 
We will now show that the result in homology is in some sense independent of these choices.

\begin{lem}\label{lem:independence_for_admissible_maps}
    Let $A_0 \subseteq A_1 \subseteq \dots \subseteq A_d \subseteq X$ and $A'_0 \subseteq A'_1 \subseteq \dots \subseteq A'_d \subseteq X$ be topological spaces with $A_i \subseteq A'_i$ for all $i$ and assume we are given admissible sequences of subspace covers $\alpha_i$ of $A_i$ and $\alpha'_i$ of $A'_i$ in $X$ such that $\alpha_i$ refines $\alpha'_i$ for all $i$.
    Further, assume we are given admissible chain maps $\lambda$ and $\lambda'$ associated to the sequences $\alpha_i$ and $\alpha'_i$, respectively.
    Then %the chain maps $\eta_{\alpha_d, \alpha_{d}'} \circ \lambda$ and $\lambda' \circ \pi_{\alpha_0, \alpha_{0}'}$ are chain homotopic, so 
    the diagram
    \[
        \begin{tikzcd}
            H_{\leq d-1}(\sing_*(\alpha_{d})) \arrow[r, "H(\eta)"] & H_{\leq d-1}(\sing_{*}(\alpha'_{d})) \\
            H_{\leq d-1}(\viet_{*}^d(\alpha_0)) \arrow[r, "H(\pi)"] \arrow[u, "H(\lambda)"] & H_{\leq d-1}(\viet_{*}^d(\alpha'_0)) \arrow[u, "H(\lambda')"]
        \end{tikzcd}
    \]
    commutes.
\end{lem}
\begin{proof}
	We prove the claim by constructing a suitable chain homotopy for the diagram above, i.e., a map $D \colon \viet_{*}^{d-1}(\alpha_{0}) \to \sing_{*}(\alpha'_{d})$ such that $D \circ \partial + \partial \circ D = \eta \circ \lambda - \lambda' \circ \pi$.
	Note that it suffices to construct $D$ on the $(d-1)$-skeleton because we only want to show commutativity of the diagram in homology up to dimension $d-1$.
	In particular, the existence of $D$ on the $(d-1)$-skeleton already implies that  $\eta \circ \lambda - \lambda' \circ \pi$ maps $(d-1)$-cycles to $(d-1)$-boundaries.
	
	As for the construction of admissible chain maps in the proof of \cref{lem:cech_to_singular}, we will perform the construction of $D$ only on oriented simplices and proceed inductively on their dimension.
	As part of our induction hypothesis, and similarly to the admissibility condition for chain maps from \cref{defi:admissibility}, we will require that for any $n$-simplex $\rho$ of $\Vietoris^{d-1}(\alpha_{0})$ there exists $U \in \alpha'_{n+1}$ such that the support of the singular $(n+1)$-chain $D(\rho)$, as well as the vertices of $\rho$ are contained in $U$.
	
	For a $0$-simplex $\{v\}$, we set $D(\{v\}) = 0$, which clearly satisfies the required conditions because $\lambda(\{v\})$ and $\lambda'(\{v\})$ are both singular $0$-simplices that take value $v$.
	Now consider an oriented $1$-simplex $[v_0,v_1] \in \viet^{d-1}_*(\alpha_0)$.
	Because $\lambda$ and $\lambda'$ are admissible, there exist $U \in \alpha_{1}$ and $U' \in \alpha'_{1}$ such that $U$ contains the support of $\eta(\lambda([v_{0}, v_{1}]))$, $U'$ contains the support of $\lambda'(\pi([v_{0}, v_{1}]))$, and $v_{0},\ v_{1} \in U \cap U'$.
	Since $\alpha_{1}$ refines $\alpha'_{1}$, we can choose $U'' \in \alpha'_{1}$ such that $U \subseteq U''$.
	In particular, we then have $U'' \cap U' \neq \emptyset$, so $U'' \cup U' \subseteq \St_{\alpha'_{1}} U'$. 
	By assumption, $\alpha'_{1}$ is an $\HLC$ star refinement of $\alpha'_{2}$, so we may choose $V \in \alpha'_{2}$ such that $\St_{\alpha'_{1}} U' \subseteq V$ and the inclusion map is trivial in singular homology.
	Since $c = \eta(\lambda([v_{0}, v_{1}])) - \lambda'(\pi([v_{0}, v_{1}]))$ is a singular 1-cycle supported in $\St_{\alpha'_{1}} U'$ (which can be verified by an easy direct computation), this implies that there is a singular 2-chain $c' \in \sing_{*}(V) \subseteq \sing_{*}(\alpha'_{2})$ such that $\partial(c') = c$.
	Setting $D([v_{0}, v_{1}]) = c'$ then satisfies the above requirements.
	In particular, we can calculate \[\partial(D([v_{0}, v_{1}])) + D(\partial([v_{0}, v_{1}])) = \partial(c') + 0 = c = \eta(\lambda([v_{0}, v_{1}])) - \lambda'(\pi([v_{0}, v_{1}])).\]
	
	Next, assume that $D$ has been defined for oriented $(n-1)$-simplices and consider an oriented $n$-simplex $\rho = [v_0,\dots,v_n] \in \viet^{d-1}_*(\alpha_0)$, where $2 \leq n$, with boundary components $\rho_{i}$.
	As before, we may use the admissibility of $\lambda$ and $\lambda'$ to choose $U \in \alpha_{n}$ and $U' \in \alpha'_{n}$ such that $U$ contains the support of $\eta(\lambda(\rho))$, $U'$ contains the support of $\lambda'(\pi(\rho))$, and both $U$ and $U'$ contain the vertices of $\rho$.
	Again, we choose some $U'' \in \alpha'_{n}$ such that $U \subseteq U''$.
	Using the induction hypothesis, we may also choose $U_{i} \in \alpha'_{n}$ such that the support of $D(\rho_{i})$ and all of the vertices of $\rho_{i}$ are contained in $U_{i}$ for all $i$.
	Because we assume $n \geq 2$ we can always choose $j \notin \{0, i\}$ to obtain $v_j \in U_0 \cap U_i$, so we must have $U_0 \cap U_i \neq \emptyset$ for all $i$.
	It follows that $U_i \subseteq \St_{\alpha'_{n}} U_0$.
	We also have that $v_{n} \in U'' \cap U_{0}$ and $v_{n} \in U' \cap U_{0}$, so we get $U'', U' \subseteq \St_{\alpha'_{n}} U_0$, too.
	In total, we obtain that the singular $n$-cycle $c = \eta(\lambda(\rho)) - \lambda'(\pi(\rho)) - D(\partial(\rho))$
	%\partial(c) = \eta(\lambda(\partial(\rho))) - \lambda'(\pi(\partial(\rho))) - \partial(D(\partial(rho))) = D(\partial(\partial(\rho))) + \partial(D(\partial(\rho))) - \partial(D(\partial(\rho))) = 0
	has support in $\St_{\alpha'_{n}} U_0$.
	Since $\alpha'_{n}$ is an $\HLC$ star refinement of $\alpha'_{n+1}$ we may now choose $V \in \alpha'_{n+1}$ such that $\St_{\alpha'_{n}} U_0 \subseteq V$ and the corresponding inclusion map is trivial in singular homology.
	Thus, there is some $(n+1)$-chain $c' \in \sing_{*}(V) \subseteq \sing_{*}(\alpha'_{n+1})$ with $\partial(c') = c$.
	Setting $D(\rho) = c'$, we calculate \[\partial(D(\rho)) + D(\partial(\rho)) = \partial(c') + D(\partial(\rho)) = c + \eta(\lambda(\rho)) - \lambda'(\pi(\rho)) - c = \eta(\lambda(\rho)) - \lambda'(\pi(\rho)),\] which finishes the proof.
\end{proof}

Note that \cref{lem:independence_for_admissible_maps} can in particular be applied in the case where $A_i = A'_i$ and $\alpha_i = \alpha'_i$, which implies that the maps $\eta$ and $\pi$ are identities and hence two choices of associated chain maps $\lambda$ and $\lambda'$ for the same sequence of covers yield identical maps in homology up to dimension $d-1$.

Coming back to our setting of functions $f \colon X \to \R$, we now consider a dimension $d$ and indices $s < t$ satisfying $t - s > d \cdot \lcs_{\mathrm{singular}}(f)$ to define

\begin{align*}
    \indexSet_{s,t,d} = \{ (t_0,\dots,t_{d},\alpha_0,\dots,\alpha_{d}) \mid &
                            t_i \in \R, %\\ &
                            t_0 = s, %\\ &
                            t_d = t, \\ &
                            t_{i+1} - t_i > \lcs_{\mathrm{singular}}(f)\text{ for all }i, \\ &
                            (\alpha_i)_i \text{ admissible covers for } f_{\leq t_0} \subseteq \dots \subseteq f_{\leq t_d} \subseteq X
                        \}.
\end{align*}

We define an order $\leq$ on $\indexSet_{s,t,d}$ by saying that $(t_0,\dots,\alpha_{d}) \leq (t'_0,\dots,\alpha'_{d})$ if and only if $f_{\leq t_i} \subseteq f_{\leq t'_i}$ and $\alpha_i$ refines $\alpha'_i$ for all $i$.

\begin{lem}\label{lem:directedness}
    Let $f \colon X \to \R$ be a function whose sublevel sets are paracompact Hausdorff spaces, and consider a dimension $d$ and indices $s < t$ satisfying $t - s > d \cdot \lcs_{\mathrm{singular}}(f)$.
    Then $(\indexSet_{s,t,d},\leq)$ is a non-empty directed set.
\end{lem}
\begin{proof}
	It is clear that the relation $\leq$ defines a preorder because the refinement order for covers of a single space defines a preorder.
	
	For non-emptyness, we proceed inductively. 
	Start with an arbitrary choice of $t_{i}$ with $t_0 = s$, $t_d = t$, $t_{i+1} - t_i > \lcs_{\mathrm{singular}}(f)$, and an arbitrary choice of $\alpha_{d} \in \Cov(f_{\leq t})$. 
	Next, given $\alpha_{i+1}$, choose $\alpha'_i$ as a cover of $f_{\leq t_{i}}$ such that for every $U' \in \alpha'_i$ there exists $V \in \alpha_{i+1}$ with $U' \subseteq V$ and such that the inclusion $f_{\leq t_i} \cap U' \to f_{\leq t_{i+1}} \cap V$ is trivial for $\tilde{H}$.
	This choice is possible since $t_{i+1} - t_i > \lcs_{\mathrm{singular}}(f)$.
	Next, choose a star refinement $\alpha_i$ of $\alpha'_i$, which is possible because $f_{\leq t_{i}}$ is assumed to be a paracompact Hausdorff space \cite{Engelking.1989}.
	Clearly, $\alpha_i$ is then an $\HLC$ star refinement of $\alpha_{i+1}$.
	
	To show directedness, we consider elements $(t_0,\dots,\alpha_{d})$ and $(t'_0,\dots,\alpha'_{d})$ in $\indexSet_{s,t,d}$ and construct a common lower bound $(t''_0,\dots,\alpha''_{d})$ for them as follows.
	First, we set $t''_{i} = \min \{t_{i}, t'_{i}\}$, which clearly satisfies $t''_{i + 1} - t''_{i} > \lcs_{\mathrm{singular}}(f)$ for all $i$.
	Next, we choose $\alpha''_{d} \in \Cov(f_{\leq t''_{d}})$ as an arbitrary common refinement of $\alpha_{d}$ and $\alpha'_{d}$.
	Finally, we inductively define $\alpha''_{i} \in \Cov(f_{\leq t''_{i}})$ by constructing an $\HLC$ star refinement $\beta_{i}$ of $\alpha''_{i+1}$ as above, and then choosing $\alpha''_{i}$ as a common refinement of $\beta_{i}$, $\alpha_{i}$, and $\alpha'_{i}$.
\end{proof}

We can now define an inverse system of maps indexed by the non-empty directed set $\indexSet_{s,t,d}$ by mapping $(t_0,\dots,t_{d},\alpha_0,\dots,\alpha_{d})$ to 
\[
    \begin{tikzcd}
       H_{\leq d-1}(\viet_{*}^d(\alpha_0)) \arrow[r, "H(\lambda)"] & H_{\leq d-1}(\sing_{*}(\alpha_{d})),
    \end{tikzcd}
\]
where $\lambda$ is some choice of admissible chain map associated to the covers $\alpha_i$.
The connecting maps corresponding to the relations on $\indexSet_{s,t,d}$ are given by pairs $(H(\pi),H(\eta))$, which is well-defined by \cref{lem:independence_for_admissible_maps}.

Next, we want to define the map $\psi_{s,t} \colon \CH_{\leq d-1}(f_{\leq s}) \to H_{\leq d-1}(f_{\leq t})$ as the inverse limit of the above maps $H(\lambda)$ over the set $\indexSet_{s,t,d}$. 
To do so, what remains to be shown is that the limit of the domains of the $H(\lambda)$ is $\CH_{\leq d-1}(f_{\leq s})$ and that the limit of the codomains is $H_{\leq d-1}(f_{\leq t})$.
We define 
\begin{align*}
    \mathcal{C}_{s,t,d} &= \{\alpha \in \Cov(f_{\leq s}) \mid \text{ there exists } (t_0,\dots,t_d,\alpha_0,\dots,\alpha_d) \in \indexSet_{s,t,d} \text{ with } \alpha_0 = \alpha\}, \\
    \mathcal{C}'_{s,t,d} &= \{\alpha \in \Cov(f_{\leq t}) \mid \text{ there exists } (t_0,\dots,t_d,\alpha_0,\dots,\alpha_d) \in \indexSet_{s,t,d} \text{ with } \alpha_d = \alpha\}.
\end{align*}
Both sets will be considered as posets with the refinement relation.

\begin{lem}\label{lem:coinitiality}
    Let $f \colon X \to \R$ be a function whose sublevel sets are paracompact Hausdorff spaces, and consider a dimension $d$ and indices $s < t$ satisfying $t - s > d \cdot \lcs_{\mathrm{singular}}(f)$.
    Then the subsets $\mathcal{C}_{s,t,d} \subseteq \Cov(f_{\leq s})$ and $\mathcal{C}'_{s,t,d} \subseteq \Cov(f_{\leq t})$ are coinitial.
\end{lem}
\begin{proof}
	First, note that the part of the proof of \cref{lem:directedness} where non-emptyness of $\indexSet_{s,t,d}$ is shown actually establishes that $\alpha_{d}$ can be chosen arbitrarily, so that $\mathcal{C}'_{s,t,d} = \Cov(f_{\leq t})$.
	For the other assertion, let $\alpha \in \Cov(f_{\leq s})$ be an arbitrary cover.
	We choose some element $(t_0,\dots,t_{d},\alpha_0,\dots,\alpha_{d}) \in \indexSet_{s,t,d}$, which is possible because the set is non-empty.
	If $\alpha'_{0}$ is a common refinement of $\alpha$ and $\alpha_{0}$, then $\alpha'_{0}$ is clearly still an $\HLC$ star refinement of $\alpha_{1}$, so that $(t_0,\dots,t_{d},\alpha'_0,\alpha_{1}\dots,\alpha_{d})$ is an element of $\indexSet_{s,t,d}$.
	Hence, we have $\alpha'_{0} \in \mathcal{C}_{s,t,d}$ and $\alpha'_{0}$ refines $\alpha$, so $\mathcal{C}_{s,t,d} \subseteq \Cov(f_{\leq s})$ is indeed coinitial.
\end{proof}

As a consequence of \cref{lem:coinitiality}, we get that the domain of $\lim_{\indexSet_{s,t,d}} H(\lambda)$ is 
\begin{align*}
    \lim_{(t_0,\dots,t_d,\alpha_0,\dots,\alpha_d) \in \indexSet_{s,t,d}} H_{\leq d-1}(\viet_{*}^d(\alpha_0)) 
    &\cong
    \lim_{\alpha \in \mathcal{C}_{s,t,d}} H_{\leq d-1}(\viet_{*}^d(\alpha)) \\
    %&\cong
    %\lim_{\alpha \in \Cov(f_{\leq s})} H_{\leq d-1}(\viet_{*}^d(\alpha)) \\
    &\cong
    \lim_{\alpha \in \Cov(f_{\leq s})} H_{\leq d-1}(\viet_{*}(\alpha)) \\
    &\cong
    \CH_{\leq d-1}(f_{\leq s}),
\end{align*}
where we have made use of the fact that for any simplicial complex $K$, its $d$-skeleton determines its homology up to degree $d-1$, i.e., $H_{\leq d-1}(K) \cong H_{\leq d-1}(K^d)$.
Similarly, we obtain that the codomain of $\lim_{\indexSet_{s,t,d}} H(\lambda)$ is
\begin{align*}
    \lim_{(t_0,\dots,t_d,\alpha_0,\dots,\alpha_d) \in \indexSet_{s,t,d}} H_{\leq d-1}(\sing_*(\alpha_{d})) 
    &\cong
    \lim_{\alpha \in \mathcal{C}'_{s,t,d}} H_{\leq d-1}(\sing_{*}(\alpha)) \\
    &\cong
    \lim_{\alpha \in \Cov(f_{\leq t})} H_{\leq d-1}(\sing_{*}(\alpha)) \\
    &\cong
    H_{\leq d-1}(f_{\leq t}).
\end{align*}
Hence, we indeed get a well-defined map 
\[
    \psi_{s,t} = \lim_{\indexSet_{s,t,d}} H(\lambda) \colon \CH_{\leq d-1}(f_{\leq s}) \to H_{\leq d-1}(f_{\leq t})
\]
whenever $t - s > d \cdot \lcs_{\mathrm{singular}}(f)$.

We have now constructed or reviewed all maps that make up \cref{diagram:main}. Before showing commutativity of this diagram, we finish this section by showing that the construction of the maps $\psi_{s,t}$ is in some sense consistent among different choices for the indices $s$ and $t$.

\begin{prop}\label{prop:comm1}
    Let $f \colon X \to \R$ be a function whose sublevel sets are paracompact Hausdorff spaces, and consider a dimension $d$, an index $s \in \R$, and $\delta > d \cdot \lcs_{\mathrm{singular}}(f)$.
    Then the diagram 
    \begin{equation}\label{eq:psi}
        \begin{tikzcd}
                    H_{\leq d-1}(f_{\leq s+\delta}) \arrow[r] & H_{\leq d-1}(f_{\leq s+2\delta}) \\
                    \CH_{\leq d-1}(f_{\leq s}) \arrow[r] \arrow[u,"\psi_{s,s+\delta}"] & \CH_{\leq d-1}(f_{\leq s+\delta}) \arrow[u,"\psi_{s+\delta, s+2\delta}"] 
        \end{tikzcd}    
    \end{equation}
    commutes.
\end{prop}
\begin{proof}
	We write $s' = s + \delta$ and $s'' = s + 2 \delta$.
	Now, define a map $\tau \colon \indexSet_{s',s'',d} \to \indexSet_{s,s',d}$ as follows.
	Starting with an element $(t'_{0},\dots,t'_{d},\alpha'_{0},\dots,\alpha'_{d}) \in \indexSet_{s',s'',d}$, we set $t_{i} = t'_{i} - \delta$.
	We also choose $\alpha_{d} \in \Cov(f_{\leq t_{d}})$ as an arbitrary refinement of $\alpha'_{d}$.
	Then, similar to the construction of common refinements in the proof of \cref{lem:directedness}, we inductively choose $\beta_{i} \in \Cov(f_{\leq t_{i}})$ as an $\HLC$ star refinement of $\alpha_{i+1} \in \Cov(f_{\leq t_{i+1}})$, and then choose $\alpha_{i}$ as a common refinement of $\beta_{i}$ and $\alpha'_{i}$. 
	This process is possible because $t_{i + 1} - t_{i} = t'_{i + 1} - t'_{i} > \lcs_{\mathrm{singular}}(f)$.
	By construction, $(t_{0},\dots,\alpha_{d})$ is an element of $\indexSet_{s,s',d}$, so we may set $\tau(t'_{0},\dots,\alpha'_{d}) = (\tau(t_{0}),\dots,\tau(\alpha_{d})) = (t_{0},\dots,\alpha_{d})$.
	
	Since we have $f_{\tau(t'_{i})} \subseteq f_{t'_{i}}$ and $\tau(\alpha'_{i})$ refines $\alpha'_{i}$ for all $i$, we can apply \cref{lem:independence_for_admissible_maps} to obtain a commutative diagram
	\[
		\begin{tikzcd}
            		H_{\leq d-1}(\sing_*(\tau(\alpha'_{d}))) \arrow[r, "H(\eta)"] & H_{\leq d-1}(\sing_{*}(\alpha'_{d})) \\
			H_{\leq d-1}(\viet_{*}^d(\tau(\alpha'_0))) \arrow[r, "H(\pi)"] \arrow[u, "H(\lambda)"] & H_{\leq d-1}(\viet_{*}^d(\alpha'_0)) \arrow[u, "H(\lambda')"] 
        		\end{tikzcd}
	\]
	for any element $(t'_{0},\dots,t'_{d},\alpha'_{0},\dots,\alpha'_{d}) \in \indexSet_{s',s'',d}$, where $\lambda$ and $\lambda'$ are choices of admissible chain maps.
	In other words, we obtain a morphism between the inverse systems of admissible maps $H(\lambda)$ and $H(\lambda')$, which by a standard procedure gives rise to a map between their limits.
	That is, we obtain the claimed commutative \cref{eq:psi}.
	That the limits of the maps $H(\pi)$ and $H(\eta)$ yield the inclusion-induced maps is an immediate consequence of their definitions.
\end{proof}

\section{A homotopy on the singular chain complex}
\label{section:homotopy_singular}
The goal of this section is to prove that the triangle in \cref{diagram:main} involving the inclusion-induced map in singular homology commutes.

\begin{prop}\label{prop:comm2}
    Let $f \colon X \to \R$ be a function whose sublevel sets are paracompact Hausdorff spaces and consider a dimension d, as well as $\delta > d \cdot \lcs_{\text{singular}}(f)$. 
    Then for all $s \in \R$ the diagram
    \begin{equation}\label{eq:comm1}
        \begin{tikzcd}
                    H_{\leq d-1}(f_{\leq s}) \arrow[r] \arrow[d, "\varphi_s", swap] & H_{\leq d-1}(f_{\leq s+\delta}) \\
                    \CH_{\leq d-1}(f_{\leq s})  \arrow[ru,"\psi_{s,s+\delta}", swap] &
        \end{tikzcd}
    \end{equation}
    commutes.
\end{prop}

We start with a lemma. 
It is analogous to \cite[Lemma 8]{Mardesic.1958}.

\begin{lem}\label{lem:comm2}
    Let $A_0 \subseteq A_1 \subseteq \dots \subseteq A_d \subseteq X$ be topological spaces with an admissible sequence of covers $\alpha_{i}$ and an admissible chain map $\lambda$.
    Then the diagram
    \begin{equation}\label{eq:comm2'}
        \begin{tikzcd}
                    H_{\leq d-1}(\sing_{*}(\alpha_{0})) \arrow[r, "H(\eta)"] \arrow[d, "H(\mu)",swap] & H_{\leq d-1}(\sing_{*}(\alpha_{d})) \\
                    H_{\leq d-1}(\viet^{d}_{*}(\alpha_{0})) \arrow[ru,"H(\lambda)", swap] &
        \end{tikzcd}
    \end{equation}
    commutes.
\end{lem}
\begin{proof}
	We follow the general outline of the proof of \cref{lem:independence_for_admissible_maps} and construct a chain homotopy $D \colon \sing_{*}(\alpha_{0}) \to \sing_{*}(\alpha_{d})$ for the above diagram, i.e., a map such that 
	$D \circ \partial + \partial \circ D = \lambda \circ \mu - \eta$.
	Again, it suffices to define $D$ for simplices up to dimension $d-1$, and to do so we proceed inductively on the dimension.
	As part of the induction hypothesis, we will again assume that for any $n$-simplex $\sigma$ of $\sing_{*}(\alpha_{0})$ with $n \leq d-1$ there exists $U \in \alpha'_{n+1}$ such that the singular $(n+1)$-chain $D(\sigma)$ and the singular simplex $\sigma$ are supported in $U$.
		
	On 0-simplices, we set $D$ to be 0.
	For a 1-simplex $\sigma$, we distinguish two cases: either $\mu(\sigma) \neq 0$ or $\mu(\sigma) = 0$.
	If $\mu(\sigma)$ is not 0, it is an oriented 1-simplex in $\viet_{*}(\alpha_{0})$.
	Because $\lambda$ is admissible, we may hence choose $U \in \alpha_{1}$ such that the support of $\lambda(\mu(\sigma))$ and the vertices of $\mu(\sigma)$ are contained in $U$.
	We have $\sigma \in \sing_{*}(\alpha_{0})$ and $\alpha_{0}$ refines $\alpha_{1}$, so we may also choose $U' \in \alpha_{1}$ such that $U'$ contains the support of $\sigma$. 
	Note that the support of $\sigma$ contains the vertices of $\mu(\sigma)$, so we have $U \cap U' \neq \emptyset$.
	Hence, the 1-cycle $\lambda(\mu(\sigma)) - \eta(\sigma)$ is supported in $\St_{\alpha_{1}} U$.
	Since the covers $\alpha_{i}$ are admissible, we may now choose $V \in \alpha_{2}$ such that $\St_{\alpha_{1}} U \subseteq V$ and the inclusion map is trivial in singular homology.
	It follows that there exists a 2-chain $c' \in \sing_{*}(V) \subseteq \sing_{*}(\alpha_{2})$ with $\partial(c') = c$.
	We set $D(\sigma) = c'$.
	
	If $\mu(\sigma) = 0$, we choose $U \in \alpha_{0}$ such that the image of $\sigma$ is contained in $U$.
	Since $\alpha_{0}$ is an $\HLC$ star refinement of $\alpha_{2}$, we may pick $V \in \alpha_{2}$ such that $U \subseteq V$ and the inclusion map is trivial in singular homology.
	Note that $\mu(\sigma) = 0$ implies that the two vertices of $\sigma$ are mapped to the same point in $U$. 
	In particular, this implies that $c = \sigma$ is a 1-cycle in $\sing_{*}(U)$, so there exists a 2-chain $c' \in \sing_{*}(V) \subseteq \sing_{*}(\alpha_{2})$ with $\partial(c') = c$.
	We set $D(\sigma) = c'$.
	
	Now, assume that $D$ has been defined as required for simplices up to dimension $n-1$ and let $\sigma$ be an $n$-simplex in $\sing_{*}(\alpha_{0})$, where $2 \leq n \leq d-1$.
	Again, we distinguish between the cases where either $\mu(\sigma) \neq 0$ or $\mu(\sigma) = 0$.
	If $\mu(\sigma)$ is not 0, it is an oriented $n$-simplex of $\viet_{*}(\alpha_{0})$.
	Using admissibility of $\lambda$, we may consequently choose $U \in \alpha_{n}$ such that the support of $\lambda(\mu(\sigma))$ and the vertices of $\mu(\sigma)$ are contained in $U$.
	We may also choose $U' \in \alpha_{n}$ such that the support of $\sigma$ is contained in $U'$. 
	With the induction hypothesis on $D$, it is also possible to choose $U_{i} \in \alpha_{n}$ for all boundary simplices $\sigma_{i}$ of $\sigma$ such that $U_{i}$ contains the supports of $D(\sigma_{i})$ and $\sigma_{i}$.
	Writing $v_{0},\dots,v_{n}$ for the vertices of $\mu(\sigma)$, we obtain $v_{n} \in U \cap U' \cap U_{0}$ and $v_{j} \in U_{i} \cap U_{0}$ for $j \notin \{0,i\}$.
	Hence, the singular $n$-cycle $c = \lambda(\mu(\sigma)) - \eta(\sigma) - D(\partial(\sigma))$ is supported in $\St_{\alpha_{n}}(U_{0})$.
	Using admissibility of the $\alpha_{i}$, we can now choose $V \in \alpha_{n+1}$ such that $\St_{\alpha_{n}}(U_{0}) \subseteq V$ and the inclusion map is homologically trivial.
	Thus, there exists $c' \in \sing_{*}(V) \subseteq \sing_{*}(\alpha_{n+1})$ such that $\partial(c') = c$ and we set $D(\sigma) = c'$.
	
	If $\mu(\sigma) = 0$, we start by choosing $U \in \alpha_{n}$ such that contains the support of $\sigma$, which is possible because $\sigma \in \sing_{*}(\alpha_{0})$ and $\alpha_{0}$ refines $\alpha_{n}$.
	For every boundary simplex $\sigma_{i}$ of $\sigma$ we use the induction hypothesis to choose $U_{i} \in \alpha_{n}$ such that $U_{i}$ contains the supports of $\sigma_{i}$ and $D(\sigma_{i})$.
	A routine argument again implies that $U \cap U_{0} \neq \emptyset$ and $U_{i} \cap U_{0} \neq \emptyset$ for all $i$, so the $n$-cycle $c = \lambda(\mu(\sigma)) - \eta(\sigma) - D(\partial(\sigma)) = - \eta(\sigma) - D(\partial(\sigma))$ is supported in $\St_{\alpha_{n}}(U_{0})$.
	Since the $\alpha_{i}$ are admissible we may choose $V \in \alpha_{n+1}$ such that $\St_{\alpha_{n}}(U_{0})$ includes homologically trivially into $V$.
	This implies that there is an $(n+1)$-chain $c' \in \sing_{*}(V) \subseteq \sing_{*}(\alpha_{n+1}$ such that $\partial(c') = c$.
	Setting $D(\sigma) = c'$ finishes the construction. 
	We omit the straightforward verification that $D$ indeed has all required properties.
\end{proof}

\begin{proof}[Proof of \cref{prop:comm2}]
	Using \cref{lem:independence_for_admissible_maps}, we may consider an inverse system indexed by $\indexSet_{s,s+\delta,d}$, mapping $(t_{0},\dots,\alpha_{d})$ to the \cref{eq:comm2'},
	which commutes by \cref{lem:comm2}.
	The connecting maps in this inverse system for relations in $\indexSet_{s,s+\delta,d}$ are given by $(H(\eta), H(\eta), H(\pi))$.
	Taking the inverse limit of this system yields the commutative \cref{eq:comm1} as claimed.
\end{proof}

\section{Homotopies on \v{C}ech complexes}
\label{section:homotopy_cech}
The goal of this section is to prove that the triangle in \cref{diagram:main} involving the inclusion-induced map in \v{C}ech homology commutes.

\begin{prop}\label{prop:comm3}
    Let $f \colon X \to \R$ be a function whose sublevel sets are paracompact Hausdorff spaces and consider a dimension d, as well as $\delta > d \cdot \lcs_{\text{singular}}(f)$. 
    Then for all $s \in \R$ the diagram
    \begin{equation}\label{eq:comm3}
        \begin{tikzcd}
                    & H_{\leq d-1}(f_{\leq s+\delta}) \arrow[d, "\varphi_{s + \delta}"]  \\
                    \CH_{\leq d-1}(f_{\leq s}) \arrow[r] \arrow[ru,"\psi_{s,s+\delta}"] & \CH_{\leq d-1}(f_{\leq s+\delta})
        \end{tikzcd}
    \end{equation}
    commutes.
\end{prop}

We start with some terminology and intermediate results.

\begin{defi}
Let $X$ be a topological space and $\alpha \in \Cov(X)$.
If $\rho = [v_{0},\dots,v_{n}] \in \viet_{*}(\alpha)$ is an oriented Vietoris $n$-simplex with $U \in \alpha$ such that $v_{i} \in U$ for all $i$, and $x \in U$ is any point, we define the \emph{join} $x \vee \rho$ as the oriented Vietoris $(n+1)$-simplex $[x,v_{0},\dots,v_{n}]$. 
By convention, we set $x \vee \rho = 0$ if there is some $j$ with $x = v_{j}$.
If $\sum_{i} a_{i} \rho_{i} \in \viet_{*}(\alpha)$ is a Vietoris $n$-chain with simplices $\rho_{i}$, coefficients $a_{i}$, sets $U_{i} \in \alpha$ such that the vertices of $\rho_{i}$ are in $U_{i}$, and $x \in \bigcap_{i} U_{i}$ is any point, we define the \emph{join} $x \vee \sum_{i} a_{i} \rho_{i}$ as the $(n+1)$-chain $\sum_{i} a_{i} (x \vee \rho_{i})$.
\end{defi}

\begin{lem}\label{lem:join}
	Let $X$ be a topological space and $\alpha \in \Cov(X)$. Let $c = \sum_{i} a_{i} \rho_{i} \in \viet_{*}(\alpha)$ be an $n$-chain with $n \geq 1$ and choose sets $U_{i} \in \alpha$ such that the vertices of $\rho_{i}$ are contained in $U_{i}$. 
	Let $x \in \bigcap_{i} U_{i}$ be any point.
	Then 
	$
		\partial(x \vee c) + x \vee \partial(c) = c.
	$
\end{lem}
\begin{proof}
	Writing $\rho_{i} = [v_{i,0},\dots,v_{i,n}]$ for the simplices of $c$, 
	we can write $\partial(\rho_{i}) = \sum_{k = 0}^{n} (-1)^{k} [v_{i,0},\dots,\widehat{v_{i,k}},\dots,v_{i,n}]$ because $n \geq 1$, where $\widehat{v_{i,k}}$ means that $v_{i,k}$ is excluded from the given list.
	Denoting $c' = \partial(x \vee c) + x \vee \partial(c)$, we calculate
	\begin{align*}
		c' %&= \sum_{i} a_{i} \left( \partial(x \vee \rho_{i}) + (x \vee \partial(\rho_{i})) \right) \\
								      &= \sum_{i} a_{i} \left( \partial([x,v_{i,0},\dots,v_{i,n}]) + \left(x \vee \sum_{k = 0}^{n} (-1)^{k} [v_{i,0},\dots,\widehat{v_{i,k}},\dots,v_{i,n}] \right) \right) \\
								      %&= \sum_{i} a_{i} \left( [v_{i,0},\dots,v_{i,n}] + \sum_{k = 0}^{n} (-1)^{k+1} [x,v_{i,0},\dots,\widehat{v_{i,k}},\dots,v_{i,n}] + \sum_{k = 0}^{n} (-1)^{k} [x,v_{i,0},\dots,\widehat{v_{i,k}},\dots,v_{i,n}] \right) \\
								      &= \sum_{i} a_{i} \left( [v_{i,0},\dots,v_{i,n}] + \sum_{k = 0}^{n} \left( \left( (-1)^{k+1} + (-1)^{k} \right) [x,v_{i,0},\dots,\widehat{v_{i,k}},\dots,v_{i,n}] \right) \right) \\ 
								      %&= \sum_{i} a_{i} [v_{i,0},\dots,v_{i,n}] \\
								      %&= \sum_{i} a_{i} \rho_{i} \\
								      &= c.					    
	\qedhere
	\end{align*}
\end{proof}

The following lemma is analogous to \cite[Lemma 9]{Mardesic.1958}.

\begin{lem}\label{lem:comm3}
    Let $A_0 \subseteq A_1 \subseteq \dots \subseteq A_d \subseteq X$ be topological spaces with an admissible sequence of covers $\alpha_{i}$ and an admissible chain map $\lambda$.
    Then the diagram
    \begin{equation}\label{eq:comm3'}
        \begin{tikzcd}
                    & H_{\leq d-1}(\sing_{*}(\alpha_{d})) \arrow[d, "H(\mu)"]  \\
                    H_{\leq d-1}(\viet^{d}_{*}(\alpha_{0})) \arrow[r, "H(\pi)", swap] \arrow[ru,"H(\lambda)"] & H_{\leq d-1}(\viet^{d}_{*}(\alpha_{d}))
        \end{tikzcd}
    \end{equation}
    commutes.
\end{lem}
\begin{proof}
	We prove the claim by constructing a suitable chain homotopy for the diagram above, i.e., a map $D \colon \viet^{d}_{*}(\alpha_{0}) \to \viet^{d}_{*}(\alpha_{d})$ such that 
	$D \circ \partial + \partial \circ D = \mu \circ \lambda - \pi$.
	It suffices to define $D$ on simplices and extend it linearly.
	For 0-simplices $\rho = [v]$, we set $D(\rho) = 0$, which satisfies the above equation because we have $\mu(\lambda(\rho)) = [v] = \pi(\rho)$ since $\lambda(\rho)$ takes value $v$.
	If $\rho$ is an oriented $n$-simplex with $n \geq 1$, we choose $U \in \alpha_{n}$ such that the support of $\lambda(\rho)$ and the vertices of $\rho$ are contained in $U$, which is possible since $\lambda$ is admissible.
	Next, we choose an arbitrary point $x \in U$ and define $D(\rho) = x \vee \left(\mu(\lambda(\rho)) - \pi(\rho)\right)$.
	We have $n \geq 1$, so it follows from the \cref{lem:join} that 
	\begin{align*}
		D(\partial(\rho)) + \partial(D(\rho)) &= x \vee \left(\mu(\lambda(\partial(\rho))) - \pi(\partial(\rho)) \right) +  \partial \left( x \vee \left(\mu(\lambda(\rho) - \pi(\rho)\right) \right) \\
								   &= \partial(x \vee \left(\mu(\lambda(\rho)) - \pi(\rho) \right)) +  \partial \left( x \vee \left(\mu(\lambda(\rho) - \pi(\rho)\right) \right) \\
								   &= \mu(\lambda(\rho)) - \pi(\rho).
	\qedhere
	\end{align*}
\end{proof}

\begin{proof}[Proof of \cref{prop:comm3}]
	Using \cref{lem:independence_for_admissible_maps}, we may consider an inverse system indexed by $\indexSet_{s,s+\delta,d}$, mapping $(t_{0},\dots,\alpha_{d})$ to the \cref{eq:comm3'},
	which commutes by \cref{lem:comm3}.
	The connecting maps in this inverse system for relations in $\indexSet_{s,s+\delta,d}$ are given by $(H(\eta), H(\pi), H(\pi))$.
	Taking the inverse limit of this system yields the commutative \cref{eq:comm3}
	as claimed.
\end{proof}

\section{Proofs of the main theorems}\label{section:final_proof}
We summarize the results from \cref{prop:comm1,prop:comm2,,prop:comm3} in the following corollary.

\begin{cor}\label{cor:diagram}
    Let $f \colon X \to \R$ be a function whose sublevel sets are paracompact Hausdorff spaces and consider a dimension d, as well as $\delta > d \cdot \lcs_{\mathrm{singular}}(f)$. 
    Then for all $s \in \R$ there are maps $\psi_{s,s+\delta} \colon \CH_{d-1}(f_{\leq s}) \to H_*(f_{\leq s+\delta})$ such that the diagram
    \begin{equation}\label{diagram:interleaving}
        \begin{tikzcd}
                    H_{\leq d-1}(f_{\leq s}) \arrow[r] \arrow[d, "\varphi_s"] & H_{\leq d-1}(f_{\leq s+\delta}) \arrow[r] \arrow[d,"\varphi_{s+\delta}"] & H_{\leq d-1}(f_{\leq s+2\delta}) \arrow[d, "\varphi_{s+2\delta}"] \\
                    \CH_{\leq d-1}(f_{\leq s}) \arrow[r] \arrow[ru,"\psi_{s,s+\delta}", swap] & \CH_{\leq d-1}(f_{\leq s+\delta}) \arrow[r] \arrow[ru,"\psi_{s+\delta, s+2\delta}", swap] & \CH_{\leq d-1}(f_{\leq s+2\delta})
        \end{tikzcd}
    \end{equation}
    commutes.
\end{cor}

From \cref{cor:diagram} we can now deduce \cref{thm:main,thm:main_strong}.

\begin{proof}[Proof of \cref{thm:main}]
    We fix a dimension $d$ and indices $s < t$ and have to show that for $s < t$ the natural maps $\ker \varphi_s \to \ker \varphi_t$ and $\coker \varphi_s \to \coker \varphi_t$ are 0.
    Note that $\lcs_{\mathrm{singular}}(f) = 0$ because we assume the filtration of $f$ to be $\HLC$. 
    Thus, for $\delta = t - s$, we have $\delta > 0 = (d + 1) \cdot \lcs_{\mathrm{singular}}(f)$.
    So by \cref{cor:diagram} there exists a map $\psi_{s,t} \colon \CH_{\leq d}(f_{\leq s}) \to H_{\leq d}(f_{\leq t})$ such that \cref{diagram:interleaving} commutes. 
    Let $i_{s,t}$ denote the inclusion $f_{\leq s} \to f_{\leq t}$. 
    Given the above diagram, we see that for any $h \in \ker \varphi_s$, we have
    $
        H_{\leq d}(i_{s,t})(h) = \psi_{s,t}(\varphi_s(h)) = 0,
    $
    so that the natural map $\ker \varphi_s \to \ker \varphi_t$ must be 0.
    Similarly, we obtain that for any $h \in \CH_{\leq d}(f_{\leq s})$, we have
    $
       \CH_{\leq d}(i_{s,t})(h) = \varphi_t(\psi_{s,t}(h)) \in \im \varphi_t,
    $
    so that the natural map $\coker \varphi_s \to \coker \varphi_t$ must be 0.
\end{proof}

\begin{proof}[Proof of \cref{thm:main_strong}]
	We fix a dimension $d$.
	To show that the claimed inequality holds, it suffices to show that $H_{\leq d-1}(f_{\leq \bullet})$ and $\CH_{\leq d-1}(f_{\leq \bullet})$ are $\delta$-interleaved for any $\delta > d \cdot \lcs_{\mathrm{singular}}(f)$.
	For such $\delta$, we get maps $\psi_{s,s+\delta} \colon \CH_{\leq d}(f_{\leq s}) \to H_{\leq d}(f_{\leq s + \delta})$ for every $s \in \R$ from \cref{cor:diagram} such that \cref{diagram:interleaving} commutes.
	It follows that the maps $\psi_{s,s+\delta}$ form a $\delta$-interleaving together with the maps $\tilde\varphi_{s,s+\delta} \colon H_{\leq d}(f_{\leq s}) \to \CH_{\leq d}(f_{\leq s + \delta})$ given by composing $\varphi_{s}$ with the structure map $\CH_{\leq d}(f_{\leq s}) \to \CH_{\leq d}(f_{\leq s + \delta})$.
	This finishes the proof.
\end{proof}

\section{Lipschitz-continuity of the local connectedness shift}\label{sec:lcs_lipschitz}
We finish by proving that the local connectedness shift is a Lipschitz map with respect to the supremum norm.

\begin{prop}\label{prop:lcs_lipschitz}
	Let $X$ be a topological space and $H$ a functor from topological spaces to a category with a 0 object.
	Then for any functions $f,\ g \colon X \to \R$ we have
	\[
		| \lcs_{H}(f) - \lcs_{H}(g) | \leq 2 \cdot \| f - g \|_{\infty}.
	\]
\end{prop}
\begin{proof}
	We set $e = \| f - g \|_{\infty}$.
	It suffices to show that if $f_{\leq \bullet}$ is $\delta$-$\HLC$ for some $\delta \geq 0$, then $g_{\leq \bullet}$ is $(\delta + 2e)$-$\HLC$.
	So let $x \in X$ and consider indices $s,t$ with $f(x) < s \leq s + \delta + 2e < t$, as well as a neighborhood $V$ of $x$ in $X$.
	If $f$ is $\delta$-$\HLC$, then we can choose a neighborhood $U$ of $x$ such that $U \subseteq V$ and the inclusion $f_{\leq s + e} \cap U \to f_{\leq s + e + \delta} \cap V$ is taken to 0 by $H$.
	Now, note that we have $g_{\leq s} \subseteq f_{\leq s + e}$ and $f_{\leq s + e + \delta} \subseteq g_{\leq s + 2e + \delta}$ because $e = \| f - g \|_{\infty}$.
	It follows that we have a commutative diagram
	\[
		\begin{tikzcd}
			 & f_{\leq s + e} \cap U \arrow[r] & f_{\leq s + e + \delta} \cap V \arrow[dr] & \\
			 g_{\leq s} \cap U \arrow[ur]\arrow[rrr] & & & g_{\leq s + 2e + \delta} \cap V
		\end{tikzcd}
	\]
	consisting of inclusion maps. In particular, $g_{\leq s} \cap U \to g_{\leq s + 2e + \delta} \cap V$ factors through $f_{\leq s + e} \cap U \to f_{\leq s + e + \delta} \cap V$.
	The second map is taken to 0 by $H$, so the same is true for the first one.
	Hence, the sublevel set filtration of $g$ is indeed $(\delta + 2e)$-$\HLC$ if the sublevel set filtratiton of $f$ is $\delta$-$\HLC$, which proves the claim.
\end{proof}

\section*{Acknowledgments}
The author thanks Ulrich Bauer, Lucas Dahinden, Anibal Medina-Mardones, and Fabian Roll for interesting and helpful discussions on the subject of this article.
This research has been supported by the German Research Foundation (DFG) through the Cluster of Excellence EXC-2181/1 \emph{STRUCTURES}, and the Research Training Group RTG 2229 \emph{Asymptotic Invariants and Limits of Groups and Spaces}.

\providecommand{\bysame}{\leavevmode\hbox to3em{\hrulefill}\thinspace}
\providecommand{\MR}{\relax\ifhmode\unskip\space\fi MR }
\providecommand{\MRhref}[2]{%
  \href{http://www.ams.org/mathscinet-getitem?mr=#1}{#2}
}
\providecommand{\href}[2]{#2}

\end{document}